\documentclass[11pt]{amsart}
\usepackage[margin=1.35in]{geometry}
\usepackage{amscd,amsmath,amsxtra,amsthm,amssymb,stmaryrd,xr,mathrsfs,mathtools,enumerate,commath, comment}
\usepackage{lipsum}

\newcommand\blfootnote[1]{%
  \begingroup
  \renewcommand\thefootnote{}\footnote{#1}%
  \addtocounter{footnote}{-1}%
  \endgroup
}
\usepackage{stmaryrd}
\usepackage{xcolor}
\usepackage{commath}
\usepackage{comment}
\usepackage{tikz-cd}
\usepackage{longtable} % for 'longtable' environment
\usepackage{pdflscape} % for 'landscape' environment
\usepackage{booktabs}
\usepackage{hyperref}
\definecolor{vegasgold}{rgb}{0.77, 0.7, 0.35}
\definecolor{darkgoldenrod}{rgb}{0.72, 0.53, 0.04}
\definecolor{gold(metallic)}{rgb}{0.83, 0.69, 0.22}
\hypersetup{
 colorlinks=true,
 linkcolor=darkgoldenrod,
 filecolor=brown,      
 urlcolor=gold(metallic),
 citecolor=darkgoldenrod,
 pdftitle={A note on the distribution of Iwasawa invariants of imaginary quadratic fields},
 }

\usepackage[all,cmtip]{xy}

\DeclareFontFamily{U}{wncy}{}
\DeclareFontShape{U}{wncy}{m}{n}{<->wncyr10}{}
\DeclareSymbolFont{mcy}{U}{wncy}{m}{n}
\DeclareMathSymbol{\Sh}{\mathord}{mcy}{"58}
\usepackage[T2A,T1]{fontenc}
\usepackage[OT2,T1]{fontenc}

\newtheorem{theorem}{Theorem}[section]
\newtheorem{lem}[theorem]{Lemma}
\newtheorem{cor}[theorem]{Corollary}
\newtheorem{prop}[theorem]{Proposition}

\newtheorem{conj}[theorem]{Conjecture}

\newtheorem{rem}[theorem]{Remark}

\newcommand{\cF}{\mathcal{F}}

\newcommand{\Z}{\mathbb{Z}}

\newcommand{\Q}{\mathbb{Q}}

\newcommand{\op}[1]{\operatorname{#1}}

 \newcommand{\cFiq}{\mathcal{F}^{\operatorname{IQ}}}

\numberwithin{equation}{section}

\begin{document}

\title[Distribution of Iwasawa invariants of imaginary quadratic fields]{A note on the distribution of Iwasawa invariants of imaginary quadratic fields}

\author[A.~Ray]{Anwesh Ray}
\address[Ray]{Centre de recherches mathématiques,
Université de Montréal,
Pavillon André-Aisenstadt,
2920 Chemin de la tour,
Montréal (Québec) H3T 1J4, Canada}
\email{anwesh.ray@umontreal.ca}

\maketitle

\begin{abstract}

Given an odd prime number $p$ and an imaginary quadratic field $K$, we establish a relationship between the $p$-rank of the class group of $K$, and the classical $\lambda$-invariant of the cyclotomic $\mathbb{Z}_p$-extension of $K$. Exploiting this relationship, we prove statistical results for the distribution of $\lambda$-invariants for imaginary quadratic fields ordered according to their discriminant. Some of our results are conditional since they rely on the original Cohen--Lenstra heuristics for the distribution of the $p$-parts of class groups of imaginary quadratic fields. Some results are unconditional results ad are obtained by leveraging theorems of Byeon, Craig and others.
\end{abstract}
\blfootnote{Competing interest declaration: The author is employed at the Centre de recherches mathematiques Montreal, advised by Matilde Lalin and Antonio Lei.}
\section{Introduction}\label{s:1}
\par Throughout, $p$ will denote an odd prime number and $K$ an imaginary quadratic field. We denote by $\Z_p$ the ring of $p$-adic integers. An infinite Galois extension $K_\infty/K$ is said to be a \emph{$\Z_p$-extension} if the Galois group $\op{Gal}(K_\infty/K)$ is isomorphic to $\Z_p$ as a topological group. There is one $\Z_p$-extension of classical interest, namely the \emph{cyclotomic $\Z_p$-extension} of $K$, which we denote by $K_{\op{cyc}}$. It is defined to be the unique $\Z_p$-extension of $K$ which is contained in the infinite cyclotomic extension of $K$ generated by the $p$-power roots of unity $\mu_{p^\infty}$. We set $K_n$ to denote the \emph{$n$-th layer}, i.e., the subextension of $K_{\op{cyc}}$ such that $[K_n:K]=p^n$. Iwasawa studied the variation of $p$-primary parts of class groups up the cyclotomic tower. More precisely, let $h_{K_n}$ be the class number of $K_n$ and $e_n$ be the maximal power of $p$ that divides $h_{K_n}$. Iwasawa showed that there exist invariants $\mu_p(K), \lambda_p(K)\in \Z_{\geq 0}$ and $\nu_p(K)\in \Z$ such that \[e_n=\mu_p(K) p^n+\lambda_p(K)n+\nu_p(K)\] for $n\gg 0$ (cf. \cite{iwasawa1973zl}). The celebrated result of Ferrero and Washington (cf. \cite{ferrero1979iwasawa}) shows that $\mu_p(K)=0$. In fact, this conclusion is shown to hold for all abelian number field extensions $K/\Q$.
\par In this note, we study the distribution of the Iwasawa $\lambda$-invariant for a fixed prime $p$ and for varying imaginary quadratic fields $K$ ordered by their discriminant. The results are proven by establishing a link between the \emph{$p$-rank} of the class group of $K$ and the $\lambda$-invariant (cf. Lemma \ref{lemma link}). Cohen and Lenstra \cite{cohen1984heuristics} made predictions for the distribution for the $p$-Sylow subgroup of imaginary quadratic fields. These predictions are supported by random matrix heuristics, cf. \cite{washington1986some}, and have led to many significant research directions in the field of \emph{arithmetic statistics}. We note that heuristics modelling the distribution of classical Iwasawa invariants of imaginary quadratic fields in which $p$ does not split were formulated by Ellenberg-Jain-Venkatesh \cite{ellenberg2011modeling}. Such heuristics are based on the prediction that the $p$-adic characteristic power series associated to the Iwasawa module structure of class group towers behave randomly as $K$ varies over all imaginary quadratic fields (in which $p$ does not split). There is some numerical evidence in support of such heuristics for the prime $p=3$ (cf. the table on p.3 in \emph{loc. cit.}), however, it should be noted that obtaining such data for Iwasawa invariants of imaginary quadratic fields is more difficult than obtaining data for their class groups. Our results demonstrate that the original Cohen--Lenstra heuristics for class groups have direct implications towards modelling the behavior of $\lambda_p(K)$. Given a family of imaginary quadratic fields $\cF$, and a positive real number $x>0$, let $\cF(x)$ consist of all $K\in \cF$ for which the absolute discriminant $|\Delta_K|$ is $\leq x$. It is easy to see that the set $\cF(x)$ is finite. Let $\cFiq$ be the family of all imaginary quadratic fields. The \emph{density} of fields in a family $\cF$ is the following limit, provided it exists
\[\mathfrak{d}(\cF):=\lim_{x\rightarrow \infty} \left(\frac{\# \cF(x)}{\# \cF^{\op{IQ}}(x)}\right). \] The upper and lower densities are defined respectively as follows
\[\overline{\mathfrak{d}}(\cF):=\limsup_{x\rightarrow \infty} \left(\frac{\# \cF(x)}{\# \cF^{\op{IQ}}(x)}\right)\text{ and }\underline{\mathfrak{d}}(\cF):=\liminf_{x\rightarrow \infty} \left(\frac{\# \cF(x)}{\# \cF^{\op{IQ}}(x)}\right).\]
 Note that the upper and lower densities do exist unconditionally. Given a natural number $n$, we show that the conjecture of Cohen and Lenstra (cf. Conjecture \ref{conjecture 0}) implies an explicit lower bound for the lower density of the set of imaginary quadratic fields $K$ for which $\lambda_p(K)\geq n$. Let $\cF(p, \lambda\geq n)$ be the family of imaginary quadratic fields $K$ for which $\lambda_p(K)\geq n$. We set $\underline{\mathfrak{d}}(p, \lambda\geq n)$ to denote the lower density of $\cF(p, \lambda\geq n)$. Theorem \ref{main theorem}, which is the main result of the article, asserts that
 \[\underline{\mathfrak{d}}(p, \lambda\geq n)\geq 1-\prod_{i=1}^\infty \left(1-p^{-i}\right)\left(1+\sum_{j=1}^{n-1} \left(p^{-j^2}\prod_{k=1}^j\left(1-p^{-k}\right)^{-2}\right)\right),\] provided the above mentioned conjecture holds. 

\par In addition, we are able to prove a number of new unconditional results as well, cf. Theorem \ref{unconditional}. We show that for any finite set of odd primes $S$, there are infinitely many imaginary quadratic fields $K$ for which $\lambda_p(K)\geq 2$ for all $p\in S$. Furthermore, it is shown that there are infinitely many imaginary quadratic fields $K$ such that $\lambda_3(K)\geq 4$. These consequences are proven by leveraging unconditional results for the distribution of $p$-ranks of class groups of imaginary quadratic fields proven by Byeon \cite{byeon2006imaginary} and Craig \cite{craig1977construction}. There is considerable interest in the study of the distribution of $p$-ranks of class groups of imaginary quadratic fields, cf. the recent preprints \cite{kulkarni2021hilbert, chattopadhyay2021p}. The relationship between the $\lambda$-invariant and $p$-ranks of class groups established in this note shows that the aforementioned results have interesting implications to Iwasawa theory of class groups.

\par We describe some related work here. Horie proved (unconditionally) the infinitude of trivial $\lambda$-invariants $\lambda_p(K)$ for imaginary quadratic fields $K$, cf. \cite[Theorem 2]{horie1987note}. In greater detail, Horie showed that for any prime number $p$, there is an infinitude of imaginary quadratic fields $K$ in which $p$ is inert and $\lambda_p(K)=0$. Of related interest is the work of Jochnowitz \cite{jochnowitz1994p}, whose results have implications to the Iwasawa $\lambda$-invariants of $p$-adic L-functions over imaginary quadratic fields $K$. On the other hand, Sands \cite{sands1993non} showed that for any fixed odd prime number $p$, there is an infinitude of imaginary quadratic fields $K$ in which $p$ splits and $\lambda_p(K)\geq 2$. Note that for $p=3$, our method shows that $\lambda_3(K)\geq 4$ for infinitely many imaginary quadratic fields $K$. The Theorem \ref{unconditional} should also be seen as a refinement of Sands' result since it established simultaneous non-triviality of $\lambda_p(K)$ for a fixed set of odd primes $p$, see the Remark \ref{remark}.

\subsection*{Acknowledgment} The author's research is supported by the CRM Simons postdoctoral fellowship. He thanks the referee for helpful suggestions.

\section{Preliminaries}
This section is dedicated to setting up notation, introducing preliminary notions and proving some basic results.
\subsection{Distribution of $p$-ranks of class groups of imaginary quadratic fields}
\par Let $K$ be an imaginary quadratic field, set $\Delta_K$ to denote its discriminant. Fix a prime number $p$ and assume that $p$ is odd. Denote by $\op{Cl}(K)$ the class group of $K$, and $h_K$ the class number. The \emph{$p$-rank} of $\op{Cl}(K)$ is the dimension of $\frac{\op{Cl}(K)}{p\op{Cl}(K)}$ as a vector space over $\Z/p\Z$. We denote the $p$-rank of $K$ by $\op{r}_p(K)$. Set $\op{Cl}_p(K)$ to denote the $p$-Sylow subgroup of $\op{Cl}(K)$. We refer to a set of imaginary quadratic fields $\cF$ with a common distinguishing property as a \emph{family}. Given a family of imaginary quadratic fields $\cF$, and a positive real number $x>0$, set $\cF(x)$ to consist of all $K\in \cF$ such that $|\Delta_K|\leq x$.
\par Let $\cFiq$ be the family of all imaginary quadratic fields. Recall that $\mathfrak{d}(\cF)$ denotes the density of $\cF$, provided it exists, as defined below
\[\mathfrak{d}(\cF):=\lim_{x\rightarrow \infty} \left(\frac{\# \cF(x)}{\# \cF^{\op{IQ}}(x)}\right). \]
Given a prime $p$, denote by $\cF_{p,n}$ (resp. $\cF_{p,\geq n}$) the family of all imaginary quadratic fields $F$ for which $r_p(K)=n$ (resp. $r_p(K)\geq n$).
 
 \par Cohen and Lenstra give a model that predicts the distribution of class groups of imaginary quadratic fields. Given a finite abelian $p$-group $B$, set $\cF_B$ to consist of the imaginary quadratic fields with $\op{Cl}_p(K)$ isomorphic to $B$. 
 \begin{conj}[Cohen--Lenstra]\label{conjecture 0}
 The density $\mathfrak{d}(\cF_B)$ exists and is given by
 \[\mathfrak{d}(\cF_B)=\frac{\prod_{k=1}^\infty \left(1-p^{-k}\right)}{\# \op{Aut}(B)},\]where $\op{Aut}(B)$ denotes the automorphism group of $B$.
 \end{conj}
 The following prediction is seen to follow as a consequence, cf. \cite[(C2),(C5), p.56]{cohen1984heuristics}.
 \begin{conj}\label{conjecture}
 The densities $\mathfrak{d}(\cF_{p,\geq 1})$ and $\mathfrak{d}(\cF_{p,n})$ exist and are given by 
 \[\begin{split}
     &\mathfrak{d}(\cF_{p,\geq 1})=1-\prod_{i=1}^\infty \left(1-p^{-i}\right),\\
     &\mathfrak{d}(\cF_{p,n})=p^{-n^2}\prod_{i=1}^\infty \left(1-p^{-i}\right)\times \prod_{k=1}^n \left(1-p^{-k}\right)^{-2}.
 \end{split}\]
 \end{conj}
Note that Conjecture \ref{conjecture} implies that $\mathfrak{d}(\cF_{p,\geq n})$ exists and is given by 
\begin{equation}\label{e1}\begin{split}&\mathfrak{d}(\cF_{\geq n})=\mathfrak{d}(\cF_{p,\geq 1})-\sum_{j=1}^{n-1} \mathfrak{d}(\cF_{p,j})\\
=& 1-\prod_{i=1}^\infty \left(1-p^{-i}\right)\left(1+\sum_{j=1}^{n-1} \left(p^{-j^2}\prod_{k=1}^j\left(1-p^{-k}\right)^{-2}\right)\right). \end{split}\end{equation}

\subsection{Iwasawa invariants}
\par Let $K$ be an imaginary quadratic field and $p$ be a prime number. Set $K_{\op{cyc}}$ to denote the cyclotomic $\Z_p$-extension of $K$ and let $K_n$ be it's $n$-th layer. Assume that $p$ is odd. Let $L_n$ (resp. $L_\infty$) be the maximal unramified abelian pro-$p$ extension of $K_n$ (resp. $K_{\op{cyc}}$). Set $X_n$ (resp. $X_\infty$) denote the Galois group $\op{Gal}(L_n/K_n)$ (resp. $\op{Gal}(L_\infty/K_{\op{cyc}}))$. Given $m\geq n$, there is a natural map $X_m\rightarrow X_n$, which factors as follows
\[X_m\rightarrow \op{Gal}(K_m \cdot L_n/K_m)\rightarrow \op{Gal}(L_n/K_m\cap L_n)\rightarrow X_n.\] Note that the map $X_m\rightarrow X_n$ is surjective if $K_m\cap L_n=K_n$. We identify $X_\infty$ with the inverse limit $\varprojlim_n X_n$. Let $\Gamma_n$ (resp. $\Gamma$) denote the Galois group $\op{Gal}(K_n/K)$ (resp. $\op{Gal}(K_{\op{cyc}}/K)$); note that $X_n$ (resp. $X_\infty$) is a module over $\Gamma_n$ (resp. $\Gamma$). By Class field theory, $X_n$ is isomorphic to $\op{Cl}_p(K_n)$. 
\par The \emph{Iwasawa algebra} $\Lambda(\Gamma)$ is the completed pro-$p$ algebra 
\[\Lambda(\Gamma):=\varprojlim_n \Z_p\left[\Gamma_n\right].\]
As is well known, $X_\infty$ is a finitely generated and torsion module over $\Lambda(\Gamma)$ (cf. \cite[Lemma 13.18, Proposition 13.20, Lemma 13.21]{washington1997introduction}). A map of $\Lambda(\Gamma)$-modules $M_1\rightarrow M_2$ is a \emph{pseudo-isomorphism} if its kernel and cokernel are finite groups.

\par Fix a topological generator $\gamma$ of $\Gamma$, and identify $\Lambda(\Gamma)$ with the power series ring $\Z_p\llbracket T \rrbracket$ letting $T:=(\gamma-1)$. A polynomial $f(T)\in \Z_p\llbracket T\rrbracket$ is said to be \emph{distinguished} if it is a monic polynomial whose non-leading coefficients are units in $\Z_p$. According to the structure theorem of finitely generated torsion $\Lambda(\Gamma)$-modules (cf. \cite[Theorem 3.12]{washington1997introduction}), there is a pseudomorphism 
\[X_\infty \longrightarrow  \left(\bigoplus_{i=1}^s\frac{\Z_p\llbracket T\rrbracket}{(p^{\mu_i})}\right) \oplus \left(\bigoplus_{j=1}^t \frac{\Z_p\llbracket T\rrbracket}{(f_i(T)^{\lambda_i})}\right),\]
where $\mu_i, \lambda_j$ are positive integers, and $f_j(T)$ are irreducible distinguished polynomials. The $\mu$-invariant $\mu_p(K)$ is the sum of the entries $\sum_{i=1}^s \mu_i$ if $s>0$, and is set to be $0$ if $s=0$. Likewise, the $\lambda$-invariant $\lambda_p(K)$ is $\sum_{i=1}^s \lambda_i \op{deg}(f_i)$ if $s>0$, and defined to be $0$ if $t=0$.

\begin{theorem}[Ferrero--Washington]
With respect to notation above, $\mu_p(K)=0$.
\end{theorem}

Infact, the result of Ferrero and Washington (cf. \cite{ferrero1979iwasawa}) applies to all abelian number fields $K$. As a consequence, $X_\infty(K)$ is a finitely generated $\Z_p$-module of rank $\lambda_p(K)$. In fact, a stronger result is true in this setting. 
\begin{theorem}
Let $K$ be an imaginary quadratic field and $p$ be an odd prime. Then, $X_\infty(p, K)$ is isomorphic to $\Z_p^{\lambda_p(K)}$.
\end{theorem}
\begin{proof}
The above result is \cite[Corollary 13.29]{washington1997introduction}.
\end{proof}

Our main focus is to fix an odd prime $p$, and study the distribution of $\lambda_p(K)$ as $K$ varies over all imaginary quadratic fields ordered by discriminant. On the other hand, one may ask the dual question, i.e., fix an imaginary quadratic field and let $p$ vary over all odd primes and study the distribution of $\lambda_p(K)$. For the family of primes $p$ that are inert in $K$, the following well known result provides an answer.
\begin{prop}
Let $K$ be an imaginary quadratic field. Then, for all primes $p\nmid h_K$ that are inert or ramified in $K$, the module $X_\infty(K)$ is identically $0$. In particular, for all but finitely many primes $p$ that do not split in $K$, we have that $\lambda_p(K)=0$.
\end{prop}
\begin{proof}
Let $p$ be a prime which is inert or ramified in $K$. If $p\nmid h_K$, then, $p\nmid h_{K_n}$ for all $n$ (cf. \cite[Proposition 13.22]{washington1997introduction}), and thus, $X_\infty(p, K)$ is identically $0$.
\end{proof}

\section{Main results}
\par In this section, we state and prove the main results of the paper. Recall that $X_\infty$ is the Galois group $\op{Gal}(L_\infty/\op{K}_{\op{cyc}})$ defined in the previous section, and $\lambda_p(K)$ is the associated Iwasawa $\mu$ and $\lambda$ invariants. In order to emphasize the dependence on $p$ and $K$, we shall at times use $X_\infty(p,K)$ to denote $X_\infty$.

\subsection{Cohen--Lenstra heuristics for $p$-ranks and consequences for the Iwasawa $\lambda$-invariant}
\par We study consequences of Conjecture \ref{conjecture} for the distribution of $\lambda_p(K)$ for imaginary quadratic fields $K$ which are ordered according to discriminant. Given a natural number $n$, let $\cF(p, \lambda\geq n)$ be the family of imaginary quadratic fields $K$ for which $\lambda_p(K)\geq n$. We set $\underline{\mathfrak{d}}(p, \lambda\geq n)$ denote the lower density of $\cF(p, \lambda\geq n)$.

\begin{theorem}\label{main theorem}
Let $p$ be an odd prime number and $n>0$ be an integer. Assume that $\cF_{p, \geq n}$ has a density given by \eqref{e1} (as predicted by Conjecture \ref{conjecture}). Then, it follows that  
\begin{equation}\label{e2}\underline{\mathfrak{d}}(p, \lambda\geq n)\geq 1-\prod_{i=1}^\infty \left(1-p^{-i}\right)\left(1+\sum_{j=1}^{n-1} \left(p^{-j^2}\prod_{k=1}^j\left(1-p^{-k}\right)^{-2}\right)\right).\end{equation}
\end{theorem}

Note that the product $\prod_{i=1}^\infty \left(1-p^{-i}\right)$ can be expanded in terms powers of $1/p$. We find that 
\[\prod_{i=1}^\infty \left(1-p^{-i}\right)=1-p^{-1}-p^{-2}+p^{-5}+p^{-7}+\dots,\]is very close to $1-p^{-1}$.

\par Given any complex number $q$ such that $|q|<1$, we have the relation
\begin{equation}\label{3.1}
    1+\sum_{j=1}^\infty \frac{q^{j^2}}{\prod_{k=1}^j \left(1-q^k\right)^2}=\prod_{i=1}^\infty \left(1-q^{-i}\right)^{-1},
\end{equation}
cf. \cite[p.21, (2.2.9)]{andrews1998theory}. Setting $q=p^{-1}$ in \eqref{3.1}, we find that the expression
\[1-\prod_{i=1}^\infty \left(1-p^{-i}\right)\left(1+\sum_{j=1}^{n-1} \left(p^{-j^2}\prod_{k=1}^j\left(1-p^{-k}\right)^{-2}\right)\right)\] given on the right hand side of \eqref{e2} is postive but approaches $0$ as $n\rightarrow \infty$. In the following, we let $\mathfrak{M}=(p, T)$ denote the maximal ideal of $\Lambda(\Gamma)$. Given a finitely generated torsion $\Lambda(\Gamma)$-module $X$, let $g(X)$ denote the dimension of $X/\mathfrak{M} X$ over $\Lambda(\Gamma)/\mathfrak{M}$. Let $\mu(X)$ (resp. $\lambda(X)$) be the $\mu$ (resp. $\lambda$) invariants of $X$. Set $g_p(K)$ to denote $g\left(X_\infty(p,K)\right)$. 
\begin{lem}\label{newest lemma}
Let $X$ be a finitely generated torsion $\Lambda$-module. Assume that $X$ has no non-zero finite $\Lambda$-submodules and let $\mu=\mu(X)$ and $\lambda=\lambda(X)$. Then, we have that $\mu+\lambda\geq g(X)$.
\end{lem}

\begin{proof}
Let $X(p)$ be the $p$-primary torsion submodule of $X$. The quotient $X/X(p)$ has no non-zero $p$-torsion. Given a short exact sequence of finitely generated and torsion $\Lambda(\Gamma)$-modules
\[0\rightarrow X_1\rightarrow X_2\rightarrow X_3\rightarrow 0,\]
it is easy to see that $g(X_2)\leq g(X_1)+g(X_3)$, and that 
\[\begin{split}
&\mu(X_2)=\mu(X_1)+\mu(X_3),\\
&\lambda(X_2)=\lambda(X_1)+\lambda(X_3).\\
    \end{split}\]
Consider the short exact sequence 
\[0\rightarrow X(p)\rightarrow X\rightarrow X/X(p)\rightarrow 0. \] It suffices to prove the result for $X(p)$ and $X/X(p)$ separately. First, let's consider $X/X(p)$. It is a free $\Z_p$-module of rank $\lambda\left(X/X(p)\right)$, and with $\mu$-invariant equal to $0$. It is clear that $g\left(X/X(p)\right)$ is $\leq \op{rank}_{\Z_p}\left(X/X(p)\right)=\lambda\left(X/X(p)\right)$. This establishes the inequality for $X/X(p)$. Note that since $X$ contains no non-zero finite $\Lambda(\Gamma)$-submodules, it follows that $X(p)$ contains no non-zero finite $\Lambda(\Gamma)$-submodules. Let $n$ be the smallest postive integer such that $X[p^n]=X(p)$. Set $\Omega$ to denote the quotient $\Lambda(\Gamma)/p$. Consider the short exact sequence
\[0\rightarrow X[p^{n-1}]\rightarrow X(p)\rightarrow X'\rightarrow 0,\] where $X'$ is the $\Omega$-module $X(p)/X[p^{n-1}]$. It is easy to see that since $X(p)$ does not contain non-zero finite $\Lambda(\Gamma)$-submodules, it follows that $X'$ does not contain any non-zero finite $\Omega$-submodules. Indeed, if $Y$ is a finite $\Omega$ submodule of $X'$, then the inverse image of $Y$ with respect to the map $X(p)\rightarrow X'$ is a finite $\Lambda(\Gamma)$-module, and hence is equal to $0$. Since $\Omega$ is a principal ideal domain, we find that $X'$ is a direct sum of cyclic $\Omega$-modules of the form $\Omega/(f)$. Therefore, $X'$ is a free $\Omega$-module, and in particular, there is a short exact sequence of $\Lambda(\Gamma)$-modules
\[0\rightarrow X''\rightarrow X(p)\rightarrow \Omega\rightarrow 0.\] Note that $\mu(X'')=\mu(X)-1$ and $X''$ does not contain any non-zero finite $\Lambda$-submodules. The result follows by induction on the $\mu$-invariant. Assume that $\mu(X'')\leq g(X'')$. It is clear that $\mu(\Omega)=g(\Omega)=1$. Then, the inequality $\mu(X(p))\leq g(X(p))$ follows from this. Note that $\lambda(X(p))=0$, and thus, we have deduced the inequality for $X(p)$. The result follows.
\end{proof}

\begin{lem}
Let $K$ be an imaginary quadratic field and let $p$ be an odd prime. Then, we have that $\lambda_p(K)\geq g_p(K)$.   
\end{lem}
\begin{proof}
Note that according to \cite[Proposition 13.28]{washington1997introduction}, $X_\infty$ does not contain any finite and non-trivial $\Lambda(\Gamma)$-submodules. It then follows from Lemma \ref{newest lemma} that $\mu_p(K)+\lambda_p(K)\geq g_p(K)$. Since the $\mu$-invariant vanishes, the result follows from this.
\end{proof}

\begin{lem}\label{lemma link}
Let $K$ be an imaginary quadratic field and $p$ be an odd prime. Then, we have that $\lambda_p(K)\geq r_p(K)$. 
\end{lem}

\begin{proof}
Note that $p$ is totally ramified in the cyclotomic $\Z_p$-extension of $\Q$. Let $v$ be a prime of $K$ that lies above $p$. We show first that $v$ is totally ramified in the cyclotomic $\Z_p$-extension of $K$. Let $\pi$ be the prime of above $p$ in $\Q_n$ and $w$ the prime of $K_n$ which lies above $\pi$ and $v$. The ramification index of $w$ over $p$ is divisible by $p^n$ since $p$ is totally ramified in $\Q_n$. On the other hand, the ramification index of $v$ over $p$ is either $1$ or $2$. Since $p$ is odd, it follows that the ramification index of $w$ over $v$ is $p^n$, hence, $v$ in totally ramified in $K_n$ for all $n$. 

\par Since $v$ is unramified in $L_0$, it follows that $L_0\cap K_{\op{cyc}}=K$. Therefore, the natural map $X_\infty\rightarrow X_0$ is a surjection. Since the action of $\Gamma$ on $X_0$ is trivial, this induces a surjection of vector spaces over $\Lambda(\Gamma)/\mathfrak{M}$
\[X_\infty/\mathfrak{M}\rightarrow X_0/p.\]Since $X_0$ is identified with $\op{Cl}_p(K)$, we find that $r_p(K)=\op{dim}_{\Z/p\Z} X_0/p$. Thus, we have shown that $r_p(E)\leq g_p(E)$. It follows from Lemma \ref{lemma link} that $g_p(E)\leq \lambda_p(E)$. Combining both inequalities, we obtain the result. 
\end{proof}

\begin{cor}\label{cor 3.4}
Let $n$ be a natural number and $p$ be an odd prime. Then, $\cF_{p,\geq n}$ is contained in $\cF(p, \lambda\geq n)$.
\end{cor}
\begin{proof}
Let $K$ be an imaginary quadratic field in $\cF_{p,\geq n}$. Since $p$ is odd we have from Lemma \ref{lemma link} that $\lambda_p(K)\geq r_p(K)\geq n$. Therefore $K$ is contained in $\cF(p, \lambda\geq n)$, and this proves the result.
\end{proof}
We commence with the proof of Theorem \ref{main theorem}.
\begin{proof}[of Theorem \ref{main theorem}]
 We recall that Corollary \ref{cor 3.4} establishes that $\cF(p, \lambda\geq n)$ contains $\cF_{p, \geq n}$, and therefore, we deduce that $\underline{\mathfrak{d}}(p, \lambda\geq n)\geq \underline{\mathfrak{d}}(\cF_{p, \geq n})$. On the other hand, it is assumed that $\mathfrak{d}(\cF_{p, \geq n})$ exists and equals \[1-\prod_{i=1}^\infty \left(1-p^{-i}\right)\left(1+\sum_{j=1}^{n-1} \left(p^{-j^2}\prod_{k=1}^j\left(1-p^{-k}\right)^{-2}\right)\right).\] Therefore, we find that 
 \[\begin{split}\underline{\mathfrak{d}}(p, \lambda\geq n)&\geq \underline{\mathfrak{d}}(\cF_{p, \geq n})=\mathfrak{d}(\cF_{p, \geq n})\\
 &=1-\prod_{i=1}^\infty \left(1-p^{-i}\right)\left(1+\sum_{j=1}^{n-1} \left(p^{-j^2}\prod_{k=1}^j\left(1-p^{-k}\right)^{-2}\right)\right).\end{split}\]
\end{proof}

\subsection{Unconditional results}
\par In this final subsection, we deduce some unconditional results that do not assume the Cohen--Lenstra heuristic \eqref{e1} for $\cF_{p, \geq n}$. A number of authors have studied asymptotics for $\cF_{p, \geq n}(x)$, and a number of unconditional results are known cf. \cite{ craig1977construction,byeon2006imaginary, luca2008class,levin2019quadratic, yu2020imaginary}. Corollary \ref{cor 3.4} allows one to translate such results into results about the family $\cF(p, \lambda\geq n)$. For ease of notation, we shall set $\cF(p, \lambda\geq n;x)$ to denote $\cF(p, \lambda\geq n)(x)$. Given a positive integer $m>0$, let $\cF_{m, \geq n}$ be the set of all imaginary quadratic fields such that $\op{Cl}(K)$ contains $\left(\Z/m\Z\right)^n$. Given two functions $f(x)$ and $g(x)$ of $x\in [0, \infty)$, we write $f(x)\gg g(x)$ if there is a constant $C>0$ for which $f(x)\geq C g(x)$ for all $x\geq 0$.

\begin{theorem}[Byeon, Craig]\label{theorem 3.5}
Given a natural number $m$, one has that 
\begin{equation}\label{Byeon eqn}\cF_{m, \geq 2}(x)\gg x^{\frac{1}{m}-\epsilon}.\end{equation}
Furthermore, for $m=3$, $\cF_{3, \geq 4}$ is infinite.
\end{theorem}
\begin{proof}
The first result \eqref{Byeon eqn} is due to Byeon \cite{byeon2006imaginary}. For the second result, see \cite{craig1977construction}.
\end{proof}

\begin{theorem}\label{unconditional}
Let $S=\{p_1, \dots, p_k\}$ be a finite set of odd primes, and let $m$ be the product $p_1\dots p_k$. The following assertions hold.
\begin{enumerate}
    \item Let $\cF_S$ be the family of imaginary quadratic fields $K$ such that $\lambda_p(K)\geq 2$ for all $p\in S$. In this case, we have that
    \[\cF_S(x)\gg x^{\frac{1}{m}-\epsilon},\] for any choice of $\epsilon>0$.
\item Suppose that $S=\{3\}$; then there are infinitely many imaginary quadratic fields $K$ for which $\lambda_3(K)\geq 4$.
\end{enumerate} 
\end{theorem}

\begin{proof}
Corollary \ref{cor 3.4} asserts that $\cF_{p,\geq 2}$ is contained in $\cF(p, \lambda\geq 2)$ for all odd primes $p$. Note that by definition, $\cF_S$ is the intersection of the sets $\cF(p, \lambda\geq 2)$ as $p$ ranges over $S$. Therefore $\cF_{m, \geq 2}$ is contained in $\cF_S$. The result is thus a direct consequence of Corollary \ref{cor 3.4} and Theorem \ref{theorem 3.5}. 
\end{proof}

\begin{rem}\label{remark}
The above result should be contrasted to the main result in \cite{sands1993non}, in which it is shown that for any odd prime number $p$, there is an infinitude of imaginary quadratic fields $K$ such that $\lambda_p(K)\geq 2$. Theorem \ref{unconditional} is to be viewed as a refinement of this results for the following reasons.
\begin{enumerate}
    \item It applies to any finite set of primes $S$, and not just a single prime.
    \item It implies that $\lambda_3(K)\geq 4$ for an infinitude of imaginary quadratic fields $K$.
\end{enumerate}
\end{rem}

\begin{rem}
At the time this note was written, there were many interesting unpublished results that further refined Theorem \ref{theorem 3.5}, see \cite{levin2019quadratic, kulkarni2021hilbert,chattopadhyay2021p}. All of these results potentially lead to improvements and variations in the statement of Theorem \ref{unconditional}.
\end{rem}

\bibliographystyle{alpha}
\bibliography{references}
\end{document}